\documentclass[12pt,leqno]{article}
\usepackage{amsfonts}
\usepackage{amsmath, amsthm, amsfonts, amssymb, color}
\usepackage{mathrsfs}
\usepackage{color}
\newtheorem{thm}{Theorem}[section]

\newtheorem{prp}[thm]{Proposition}

\newtheorem{rem}[thm]{Remark}
\theoremstyle{definition}
\newtheorem{defn}{Definition}[section]
\newcommand{\scr}[1]{\mathscr #1}
\definecolor{wco}{rgb}{0.5,0.2,0.3}

\numberwithin{equation}{section} \theoremstyle{remark}

\newcommand{\ua}{\uparrow}

\title{{\bf Comparison theorem for neutral stochastic functional   differential
equations driven by      $G$-Brownian motion}
 }
\author{
{\bf      Fen-Fen Yang$^{a)}$, Chenggui Yuan$^{b)}$ }\\
\footnotesize{  a)Department of Mathematics, Shanghai University, Shanghai 200444, China}\\
\footnotesize{  b)Department of Mathematics, Swansea University, Swansea SA1 8EN, UK}\\
\footnotesize{ 
yangfenfen@shu.edu.cn,  C.Yuan@swansea.ac.uk}
}

\begin{document}
\allowdisplaybreaks
\def\R{\mathbb R}  \def\ff{\frac} \def\ss{\sqrt} \def\B{\mathbf
B} \def\W{\mathbb W}
\def\N{\mathbb N} \def\kk{\kappa} \def\m{{\bf m}}
\def\ee{\varepsilon}\def\ddd{D^*}
\def\dd{\delta} \def\DD{\Delta} \def\vv{\varepsilon} \def\rr{\rho}
\def\<{\langle} \def\>{\rangle} \def\GG{\Gamma} \def\gg{\gamma}
  \def\nn{\nabla} \def\pp{\partial} \def\E{\mathbb E}
\def\d{\text{\rm{d}}} \def\bb{\beta} \def\aa{\alpha} \def\D{\scr D}
  \def\si{\sigma} \def\ess{\text{\rm{ess}}}
\def\beg{\begin} \def\beq{\begin{equation}}  \def\F{\scr F}
\def\Ric{\text{\rm{Ric}}} \def\Hess{\text{\rm{Hess}}}
\def\e{\text{\rm{e}}} \def\ua{\underline a} \def\OO{\Omega}  \def\oo{\omega}
 \def\tt{\tilde} \def\Ric{\text{\rm{Ric}}}
\def\cut{\text{\rm{cut}}} \def\P{\mathbb P} \def\ifn{I_n(f^{\bigotimes n})}
\def\C{\scr C}      \def\aaa{\mathbf{r}}     \def\r{r}
\def\gap{\text{\rm{gap}}} \def\prr{\pi_{{\bf m},\varrho}}  \def\r{\mathbf r}
\def\Z{\mathbb Z} \def\vrr{\varrho} \def\ll{\lambda}
\def\L{\scr L}\def\Tt{\tt} \def\TT{\tt}\def\II{\mathbb I}
\def\i{{\rm in}}\def\Sect{{\rm Sect}}  \def\H{\mathbb H}
\def\M{\scr M}\def\Q{\mathbb Q} \def\texto{\text{o}} \def\LL{\Lambda}
\def\Rank{{\rm Rank}} \def\B{\scr B} \def\i{{\rm i}} \def\HR{\hat{\R}^d}
\def\to{\rightarrow}\def\l{\ell}\def\iint{\int}
\def\EE{\scr E}\def\Cut{{\rm Cut}}
\def\A{\scr A} \def\Lip{{\rm Lip}}
\def\BB{\scr B}\def\Ent{{\rm Ent}}\def\L{\scr L}
\def\R{\mathbb R}  \def\ff{\frac} \def\ss{\sqrt} \def\B{\mathbf
B}
\def\N{\mathbb N} \def\kk{\kappa} \def\m{{\bf m}}
\def\dd{\delta} \def\DD{\Delta} \def\vv{\varepsilon} \def\rr{\rho}
\def\<{\langle} \def\>{\rangle} \def\GG{\Gamma} \def\gg{\gamma}
  \def\nn{\nabla} \def\pp{\partial} \def\E{\mathbb E}
\def\d{\text{\rm{d}}} \def\bb{\beta} \def\aa{\alpha} \def\D{\scr D}
  \def\si{\sigma} \def\ess{\text{\rm{ess}}}
\def\beg{\begin} \def\beq{\begin{equation}}  \def\F{\scr F}
\def\Ric{\text{\rm{Ric}}} \def\Hess{\text{\rm{Hess}}}
\def\e{\text{\rm{e}}} \def\ua{\underline a} \def\OO{\Omega}  \def\oo{\omega}
 \def\tt{\tilde} \def\Ric{\text{\rm{Ric}}}
\def\cut{\text{\rm{cut}}} \def\P{\mathbb P} \def\ifn{I_n(f^{\bigotimes n})}
\def\C{\scr C}      \def\aaa{\mathbf{r}}     \def\r{r}
\def\gap{\text{\rm{gap}}} \def\prr{\pi_{{\bf m},\varrho}}  \def\r{\mathbf r}
\def\Z{\mathbb Z} \def\vrr{\varrho} \def\ll{\lambda}
\def\L{\scr L}\def\Tt{\tt} \def\TT{\tt}\def\II{\mathbb I}
\def\i{{\rm in}}\def\Sect{{\rm Sect}}  \def\H{\mathbb H}
\def\M{\scr M}\def\Q{\mathbb Q} \def\texto{\text{o}} \def\LL{\Lambda}
\def\Rank{{\rm Rank}} \def\B{\scr B} \def\i{{\rm i}} \def\HR{\hat{\R}^d}
\def\to{\rightarrow}\def\l{\ell}
\def\8{\infty}\def\I{1}\def\U{\scr U}
\maketitle

\begin{abstract} In this paper, we   investigate  sufficient and necessary conditions    for the comparison theorem of  neutral stochastic functional   differential
equations driven by      $G$-Brownian motion ($G$-NSFDE). 
 Moreover, the results extend the ones in the linear expectation case \cite{BJ} and  nonlinear expectation framework \cite{HYang}.
\end{abstract} \noindent
 \vskip 2cm

\section{Introduction}

Since the order preservation has been introduced for studying complex stochastic models by compared   with  a simpler ones, this property
also called as ``comparison theorem" in the literature. Two types of
comparison theorems have been extensively investigated: the distribution (weak) sense and  the pathwise (strong) sense respectively, such as
Chen-Wang \cite{CW} for multidimensional diffusion processes,  Wang \cite{W1} for superprocesses 
in the distribution (weak) sense; Ikeda-Watanabe \cite{IW} for stochastic differential
equations, Yang-Mao-Yuan \cite{YMY} for one-dimensional stochastic hybrid systems  in  the pathwise (strong) sense.
Moreover, the comparison theorem in
  the pathwise (strong) sense  implies the distribution (weak) sense. We refer readers to see \cite{
HW, IW, 
 PZ, YMY
 }, and the references therein.

Recently, there are some   results for comparison
theorem on stochastic differential delay equations, such as, Bao-Yuan \cite{BY} established  a comparison theorem for stochastic differential delay
equations with jumps,  Bai-Jiang \cite{BJ} proved the comparison theorem   for stochastic functional differential equations whose drift term satisfies the quasimonotone condition and diffusion term is independent of
delay. Moreover, Huang-Yuan \cite{HY}    provided  sufficient   and necessary conditions such that the order preservation holds for distribution-dependent neutral
stochastic functional differential equations.

Furthermore, under the framework of nonlinear expectation,  some sufficient condition is presented in
\cite[Theorem 7.1]{Lin} for comparison theorem of one-dimensional stochastic differential
equations driven by $G$-Brownian motion ($G$-SDEs). In addition,
Luo-Wang  \cite{
LW, LW1}    prove a comparison theorem for multidimensional $G$-SDEs.
There are some results for comparison theorem of stochastic functional differential equations  driven by $G$-Brownian ($G$-SFDEs); see \cite{HYang}, where the representation of the $G$-expectation \eqref{rep2}
  introduced in \cite{15, HP
  } plays important role in the proof of necessary condition of the comparison theorem.

Motivated by the above significant works, we shall investigate the  order preservation for neutral stochastic functional differential equations  driven by $G$-Brownian ($G$-NSFDEs). To achieve this, thanks the work completed by He-Han \cite{HH}, the authors studied existence and stability of solutions to  $G$-NSFDEs.

The paper is organized as follows. In Section 2, we recall some preliminaries on $G$-Brownian motion and its related stochastic calculus.
In Section 3, we prove the comparison theorem   for the $G$-NSFDEs.
\section{$G$-Expectation and $G$-Brownian motion}
  For a matrix $A$,  let   $A^\ast$ be its transpose and  $|A|=\sqrt{ {\rm {trace}}(AA^\ast)}$.
   Let $\mathbb{M}^m$ be the collection of all $m\times m$ matrices and $\mathbb{S}^m$ ($\mathbb{S}_+^m$) be the set of the symmetric (symmetric and positive definite) ones in $\mathbb{M}^m$.
   For    $M, \bar{M}\in \mathbb{M}^m,$
   define $\<M,\bar{M}\>=\sum_{k,l=1}^mM _{kl}\bar{M} _{kl}.$ For $X,\bar{X}\in\mathbb{S}^m$, the notation  $X\geq \bar{X}$ (resp., $X> \bar{X}$) means that  $X- \bar{X}$
 is non-negative (resp., positive) definite.
Fix two positive constants $\underline{\sigma}$ and $\bar{\sigma}$ with $\underline{\sigma}<\bar{\sigma}$, define
\begin{equation}\label{G(A)}
 G(X)=\frac{1}{2}\sup _{\gamma\in \mathbb{S}_+^m\bigcap[\underline{\sigma}^2\textbf{I}_{m\times m}, \bar{\sigma}^2\textbf{I}_{m\times m}]}\<\gamma,X\>, \ X\in\mathbb{S}^m.
\end{equation}
From the definition of the function $G$, it has the following properties (\cite[Chapter 2]{peng4}):
\begin{enumerate}
\item[(a)] (Positive homogeneity) $G(\lambda X)=\lambda G(X)$, $\lambda\geq 0,  X\in\mathbb{S}^m$.
\item[(b)] (Sub-additivity) $G(X+\bar{X})\leq G(X)+G(\bar{X})$, $G(X)-G(\bar{X})\leq G(X-\bar{X})$,\ \   $X,\bar{X}\in\mathbb{S}^m$.
\item[(c)] $|G(X)|\leq \frac{1}{2}|X|\sup _{\gamma\in \mathbb{S}_+^m\bigcap[\underline{\sigma}^2\textbf{I}_{m\times m}, \bar{\sigma}^2\textbf{I}_{m\times m}]}|\gamma|=\frac{1}{2}|X|\sqrt{m}\bar{\sigma}^2.$
\item[(d)] $G(X)-G(\bar{X})\geq \frac{\underline{\sigma}^2}{2} \mbox{trace}[X-\bar{X}],\ X\geq \bar{X}, X,\bar{X} \in \mathbb{S}^m.$
\end{enumerate}
\begin{rem}\label{coG}
(b) and (c) imply that   $G$ defined by \eqref{G(A)} is continuous under   $|\cdot|$.
\end{rem}
For convenience, we list some notations and spaces appearing in this paper as
follows.
 \begin{itemize}

  \item   $\Omega=C_0([0,\infty);\mathbb{R}^m)$, the $\mathbb{R}^m$-valued and continuous functions on $[0,\infty)$ vanishing at zero, which is endowed with the   metric
$$
\rho(\omega^1,\omega^2):=\sum_{n=1}^\infty\frac{1}{2^n}\left[\max_{t\in[0,n]}|\omega^1_t-\omega^2_t|\wedge1 \right],  \omega^1,\omega^2\in\Omega.
$$
\item   $C_{b,lip}(\mathbb{R}^{m})$ is  the set of bounded and Lipschitz continuous functions on $\mathbb{R}^{m}$.
     \item
$L_{ip}(\Omega_T)  =\{  \varphi(\omega_{t_1}, \cdot \cdot \cdot, \omega_{t_n}), n\in \mathbb{N}^+, t_1,\cdot \cdot \cdot, t_n\in [0,T],\varphi \in C_{b,lip}((\mathbb{R}^{m})^n)\}.$

 \item  $L_{ip}(\Omega)=\bigcup_{T>0}L_{ip}(\Omega_T).$
 \end{itemize}

\begin{itemize}
  \item  $\bar{\E}^G$ is the nonlinear expectation on $\Omega$ such that coordinate process $(B(t))_{t\geq 0}$, i.e. $B(t)(\omega)=\omega_t, \omega\in \Omega$  is an $m$-dimensional $G$-Brownian motion.
    \item   $\<B\>$ is the quadratic variation process of $B$.
  \item  $L_G^p(\Omega)$ is the completion of $L_{ip}(\Omega)$  under the norm $(\bar{\E}^G|\cdot|^p)^{\frac{1}{p}}$, $p\geq1$.
\item
$M_{G}^{p,0}([0,T])=
\Big\{ \sum_{j=0}^{N-1} \eta_{j} I_{[t_j, t_{j+1})}(t);
\eta_{j}\in L_{G}^p(\Omega_{t_{j}}), 0=t_0<t_1<\cdots <t_N=T \Big\}.
$
\item $M_G^p([0,T])$ is the completion of $M_G^{p,0}([0,T])$ under
 $\|\xi\|_{M_G^p([0,T])} =\left(\bar{\E}^G\int_{0}^{T}|\xi_{t}|^p\d t\right)^{\frac{1}{p}}.$
\item $\mathcal{M}$ is the collection of all probability measures on  $(\Omega, \B(\Omega))$.
 \end{itemize}
For the construction of nonlinear expectation $\bar{\E}^G$, one can refer to \cite{peng4} for more details. According to \cite{15,HP},    there exists a weakly compact subset $\mathcal{P}\subset \mathcal{M}$   such that
\begin{align}\label{rep}\bar{\E}^G[X]=\sup_{\P\in \mathcal{P}}\mathbb{E}_\P[X], \ X\in L_G^1(\Omega),\end{align}
where $\mathbb{E}_\P$ is the linear expectation under probability measure $\P \in \mathcal{P}$. $\mathcal{P}$ is called a set that represents $\bar{\E}^G$.
To see \eqref{rep} more clear, let $W^0$ be an $m$-dimensional Brownian motion on a complete filtered probability space $(\hat{\OO},\{\F_t\}_{t\geq 0},\P)$, and define
\begin{align*}\mathbb{H}:=\{\theta: \ \ &\theta\  \text{is an}\  \mathbb{M}^m\text{-valued progressively measurable}\\
& \text{stochastic process,} \
\theta_s\theta_s^\ast\in[\underline{\sigma}^2\textbf{I}_{m\times m}, \bar{\sigma}^2\textbf{I}_{m\times m}], \   s\geq 0\}.
 \end{align*}
 For any $\theta\in \mathbb{H}$, let $\P_{\theta}$ be the law of $\int_0^\cdot \theta_s \d W^0_s$. By \cite{15,HP}, taking $\mathcal{P}= \{\P_\theta, \theta\in\mathbb{H}\}$, then
\begin{align}\label{rep2}
\bar{\E}^G[X]=\sup_{\theta\in \mathbb{H}}\mathbb{E}_{\P_\theta}[X], \ X\in L_G^1(\Omega).
\end{align}
Define $\mathcal{C}(A)=\sup_{\P\in \mathcal{P}}\P(A), \ A\in \B(\Omega)$ the associated  Choquet capacity to
$\bar{\E}^G$.
We call a set $A\in\B(\Omega)$ is   polar   if $\mathcal{C}(A)=0$,  and   a property   holds  $\mathcal{C}$-quasi-surely ($\mathcal{C}$-q.s.)
if it holds outside a polar set. 
 Moreover,  for the quadratic variation process $\<B\>$,
 by \cite[Corollary 5.7]{peng4} and the property (d) of  the function $G$, we have $\mathcal{C}$-q.s.
\begin{align}\label{B}
\underline{\sigma}^2\textbf{I}_{m\times m}\leq\frac{\d}{\d t}\langle B\rangle(t)\leq\bar{\sigma}^2\textbf{I}_{m\times m}.
\end{align}

\section{$G$-NSDE}

For a fixed  constant $r_0\geq 0$, let
$\C=C([-r_0,0];\R^d)$, which   equipped with uniform norm $\|\cdot\|_\infty$.
For any continuous map $f: [-r_0,\infty)\to \R^d$  and
$t\ge 0$,  let  $f_t\in\C$ be such that $f_t(s)=f(s+t)$ for $s\in
[-r_0,0]$. Generally,    $(f_t)_{t\ge 0}$ is called the segment of $(f(t))_{t\ge -r_0}.$
Let $x=(x^1,\cdots, x^d)$, $ y=(y^1,\cdots, y^d)\in\R^d$, we call $x\le y$ if $x^i\le y^i$ holds for all $1\le i\le d.$  Now  we introduce the partial-order on $\C.$ Let $\xi=(\xi^1,\cdots,\xi^d)$, $\eta=(\eta^1,\cdots,\eta^d)\in\C$.
 \beg{itemize}
\item We call $\xi\le \eta$ if $\xi^i(s)\le \eta^i(s)$ holds for all $s\in [-r_0,0]$ and $1\le i\le d.$
    \item
     We call $\xi\le_{N} \eta$ if $\xi\le \eta$ and $\xi(0)-N(\xi)\le \eta(0)-N(\eta)$.
\item  For any $t\in [-r_0,0]$, define $(\xi\land\eta)^i(t)=\min\{\xi^i(t),\eta^i(t)\}, \ \  1\le i\le d,$ then
 $\xi\land\eta\in \C$.

\end{itemize}
Consider the following   $G$-NSFDE:
\begin{equation}\label{E1}
\begin{cases}
\d [Y(t)-N(Y_t)]= b(t,Y_t)\,\d t+ \<h(t,Y_t),\d \<B\>(t)\> +\sigma(t,Y_t)\,\d B(t),\\
  \d [\bar{Y}(t)- {N}(\bar{Y}_t)]= \bar{b}(t,\bar{Y}_t)\,\d t+ \<\bar{h}(t,\bar{Y}_t),\d \<B\>(t)\>+ {\sigma}(t,\bar{Y}_t)\,\d B(t),
\end{cases}
\end{equation}
where $N:\C\to \R^d$ is called neutral term, $h=(h^{ij})_{1\leq i,j\leq m}$ with $h^{ij}=h^{ji}$, $\bar{h}=(\bar{h}^{ij})_{1\leq i,j\leq m}$ with $\bar{h}^{ij}=\bar{h}^{ji}$, and $b,\bar{b}, N: [0,\infty)\times \C\to \R^d; h^{ij},\bar{h}^{ij}: [0,\infty)\times \C\to \R^d;
 \si,\bar{\sigma}: [0,\infty)\times \C\to \R^d\otimes\R^m$
are   measurable.

\begin{defn}For any $s\ge0$ and $\xi,\bar\xi\in\C$, we call $(Y(t),\bar Y(t))_{t\geq s}$
 solve    (\ref{E1})    if for all $t\ge s,$
\beg{equation*}
\beg{split}
Y(t)-N(Y_t) = &\xi(0)-N(\xi)+ \int_s^t  b(r,Y_r)\d r+\int_s^t \<h(r,Y_r),\d \<B\>(r)\> +\int_s^t \si(r, Y_r)\d B(r ),\\
\bar Y(t)- {N}(\bar{Y}_t) = &\bar\xi(0)- {N}(\bar\xi)+ \int_s^t \bar b(r,\bar Y_r)\d r+\int_s^t \<\bar{h}(r,\bar{Y}_r),\d \<B\>(r)\>+ \int_s^t\bar \si(r,\bar Y_r)\d B(r),
\end{split}\end{equation*} where  $(Y_t, \bar Y_t)_{t\ge s}$ is the segment process of $(Y(t), \bar Y(t))_{t\ge s-r_0}$ with
$(Y_s,\bar Y_s)= (\xi,\bar\xi)$.
\end{defn}
Throughout the paper, we make   the following assumptions.
\beg{enumerate} \item[{\bf (H1)}]
There exists an increasing function $\aa: \R_+\to \R_+$     such that for any $ t\ge 0, \xi,\eta\in \C$,
 \begin{align*}
&|b(t,\xi)- b(t,\eta)|^2+|\bar{b}(t,\xi)- \bar{b}(t,\eta)|^2+|h(t,\xi)- h(t,\eta)|^2+|\bar{h}(t,\xi)- \bar{h}(t,\eta)|^2\\
&+|\si(t,\xi)- \si(t,\eta)|^2+|\bar{\si}(t,\xi)- \bar{\si}(t,\eta)|^2\le \aa(t)\|\xi-\eta\|_{\infty}^2.
\end{align*}
\item[{\bf (H2)}] There exists a  constant $K  $ such that
 \begin{align*}
  &|b(t,0)|^2+ |\bar b(t,0)|^2+|h(t,0)|^2+ |\bar h(t,0)|^2+\|\si(t,0)\|_{HS}^{2}+|\bar\si(t,0)|^{2}
\le K,\ \ t\ge 0.
 \end{align*}
\item[{\bf (H3)}] $N(0)=0$ and $N(\xi) \leq N(\eta)$ for $\xi\leq \eta$.
 \item[{\bf (H4)}] There exists a constant $\kappa\in (0, 1)$ such that
\begin{equation*}
|N(\xi) -N(\eta)|\leq \kappa \max_{1\leq i\leq d}\|\xi^i-\eta^i\|_{\infty}, \  \xi,\eta\in \C.
\end{equation*}
\end{enumerate}
\begin{rem}\label{remboun}
According to \cite[Lemma 2.1]{HH},  {under {\bf(H1)}-{\bf(H4)},}  for any $s\ge 0$ and $\xi, \bar\xi\in \C$,  $G$-NSFDE  \eqref{E1}   has a unique solution
  $\{Y(s,\xi;t), \bar Y(s,\bar \xi;t)\}_{t\ge s-r_0}$     with $(Y_s,\bar Y_s)=(\xi,\bar\xi)$.
Moreover, the segment processes $\{Y(s,\xi)_t, \bar Y(s,\bar\xi)_t\}_{t\ge s}$ satisfy
\beq\label{BDD} \bar{\E}^G \sup_{t\in [s,T]} \big(\|Y(s,\xi)_t\|_\infty^2+ \|\bar Y(s,\bar \xi)_t\|_\infty^2\big)<\infty,\ \ T\in [s,\infty).
\end{equation}
\end{rem}

\beg{defn} The $G$-NSFDE  $(\ref{E1})$ is called $N$-order-preserving, if  for any $s\ge 0$ and $\xi, \bar\xi\in \C$ with  $\xi\le_N\bar\xi$,
it holds  $\mathcal{C}$-q.s.   $$Y(s,\xi)_t\le_N \bar Y(s,\bar\xi)_t,\ t\ge s.$$   \end{defn}
We first state the following sufficient conditions for the order preservation, which reduce back to the corresponding ones in \cite{BJ} when the noise is an $m$-dimensional standard Brownian motion  and in \cite{HY} when the system with distribution independent.

\beg{thm}\label{T1.1} Assume {\bf (H1)}-{\bf (H4)} and the following  two conditions are satisfied:
\beg{enumerate}
\item[$\textbf{(A1)}$] For any  $1\leq i\leq d$, $\xi,\eta\in\C$ with   { $\xi\leq_N \eta$ and $\xi^i(0)-N^i(\xi)=\eta^i(0)-N^i(\eta)$,}
    \begin{equation*}
b^i(t,\xi)-\bar{b}^i(t, \eta)+2G(h^i(t,\xi)-\bar{h}^i(t, \eta))\leq 0,\ \ \text{a.e.}\ t\ge 0.
    \end{equation*}
\item[$\textbf{(A2)}$] The diffusion terms  $\si= (\si^{ij})$ and  $ \bar\si=(\bar\si^{ij})$ satisfy $\si= \bar\si$.  Moreover, $\sigma^{ij}(t,\xi)$ only depends on $t$ and $\xi^{i}(0)-N^{i}(\xi), $ for any $1\leq i\leq d$, $1\leq j\leq m$, $\xi\in \C$.
\end{enumerate}
Then the $G$-NSFDE  \eqref{E1} is $N$-order-preserving.
\end{thm}

Before proceeding to the proof of Theorem \ref{T1.1}, we give the following  proposition. 
Without loss of generality, we assume $s = 0$ and omit the subscript $s$. Denote $Y(t)=Y(s,\xi;t),$ $\bar Y(t)=\bar Y(s,\bar\xi;t)$,
$Y^N(t)=(Y^{i,N}(t), \cdot\cdot\cdot, Y^{d,N}(t))=Y(t)-N(Y_t)$
and
$\bar{Y}^N(t)=(\bar Y^{i,N}(t), \cdot\cdot\cdot, \bar Y^{d,N}(t))=\bar{Y}(t)-N(\bar{Y}_t)$. 
 Let
\begin{equation}\begin{split}
\tau^i&=\inf\{t>0: Y^i(t)>\bar{Y}^i(t)\}, i=1, 2, \cdot  \cdot \cdot, d, \\
\label{tau_D}\tau_N^i&=\inf\{t>0: Y^{i,N}(t)>\bar{Y}^{i,N}(t)\}, i=1, 2, \cdot  \cdot \cdot, d.
\end{split}\end{equation}
\begin{prp}\label{prop1} (\cite[Proposition 4.3]{HY}) Let
$\tau=\min\{ \tau^1, \cdot  \cdot \cdot, \tau^d\}$ and
$\tau_N=\min \{ \tau_N^1, \cdot  \cdot \cdot, \tau^d_N\}.$
Assume  {\bf (H3)} and {\bf (H4)}, then
    $\tau_N\leq \tau  \ {\rm{on}} \ \Omega.$
\end{prp}
\begin{proof} As we set conditions on the neutral term $N$ is same with    \cite{HY}, 
  we omit the proof.
\end{proof}
\begin{rem}\label{rem1} (\cite[Remark 4.4]{HY}) The following condition {\bf(C1)} and {\bf(C2)} are equivalent.
\begin{itemize}
  \item [\bf(C1)] {\bf(H1)} together with Theorem \ref{T1.1}   $\textbf{(A2)}$.
  \item  [\bf(C2)] $\si=\bar{\si}$, and there exists a constant $L>0$ such that for any
   $i = 1,\cdot\cdot\cdot, n,$
  \begin{align*}
     \sum_{j=1}^m&\left(| {\si}_{ij}(t,\xi)- {\si}_{ij}(t,\eta)|^2+|\bar{\si}_{ij}(t,\xi)- \bar{\si}_{ij}(t,\eta)|^2\right)\\
     &\le L|\xi^i(0)-N^i(\xi)-\eta^i(0)+N^i(\eta)|^2, \ t\geq 0, \ \xi, \eta\in\C.
  \end{align*}
\end{itemize}
\end{rem}
\begin{proof}[$\textbf{Proof of Theorem \ref{T1.1}}$]
We first prove the result under  the following condition $\textbf{(A1')}$ instead of  $\textbf{(A1)}$.
 \beg{enumerate}\item[$\textbf{(A1')}$] For any  $1\leq i\leq d$, $\xi,\eta\in\C$ with $\xi\leq_N \eta$ and $\xi^i(0)-N^i(\xi)=\eta^i(0)-N^i(\eta)$,
 \begin{equation*}
  b^i(t,\xi)-\bar{b}^i(t, \eta)+2G(h^i(t,\xi)-\bar{h}^i(t, \eta))<0,\ \ \text{a.e.}\ t\ge 0.
\end{equation*}
\end{enumerate}
Assume  {\bf(H1)}-{\bf(H4)}, and let conditions $\textbf{(A1')}$ and $\textbf{(A2)}$ hold.
We aim to prove that
$ \mathcal{C}\{Y_t\leq_N\bar Y_t,\ \ t\in[0,T]\}=1,$
which equal to prove that for any $1\le i\le d$,
\begin{equation}\label{r1}
 \mathcal{C}\{Y^{i,N}(t)>\bar Y^{i,N}(t),\ \ t\in[0,T]\}=0.
\end{equation}
By the definition of $ \tau_N^i, $ this equal to    prove that for any $1\le i\le d$,
$\mathcal{C}\{\tau_N^i=\infty \}=1.$
It suffices to prove
$\mathcal{C}\{\tau_N =\infty \}=1.$
In fact, if this holds, by Proposition \ref{prop1},  it holds that
$\mathcal{C}\{\tau=\infty 
\}=1,$
i.e., for any $1\le i\le d$,
\begin{equation}\label{r2}
     \mathcal{C}\{Y^i(t)>\bar Y^i(t),\ \ t\in[0,T]\}=0.
\end{equation}
 Combining with \eqref{r1}  and  \eqref{r2}, so the $N$-order preservation holds.
To this end, let us assume the contrary. If there exists some $M>0$  such that
$\mathcal{C}\{\tau_N<M\}>0.$ It follows the definition of  $\tau_N$ that
there exists an $i\in \{ 1, \cdot  \cdot \cdot, d\}$ such that
$\mathcal{C}(A)>0, \ {\rm{for}} \  A=\{\tau_N^i=\tau_N<M\}.$
For any $t\geq0$, let
$Y_t^Z=({Y}^{1}_{t}, \cdot\cdot\cdot,{Y}^{i-1}_{t}, {Y}^{i}_{t}-(Z^{i,N}_t)^+, {Y}^{i+1}_{t},\cdot\cdot\cdot, {Y}^{d}_{t}),$ where $(Z^{i,N}_t)^+ = \{(Y_t^{i,N}-\bar Y_t^{i,N}) \vee 0\}.$
Set
 \begin{align}\label{tauH}
 \tau_H^i&=\inf\{t>\tau_N^i: H^i(t)>0\},
 \end{align}
 where $H^i(t)=b^i(t,  Y_t^Z)-\bar b^i(t,\bar Y_t)
 +2G[h^i(t,Y_t^Z)-\bar h^i(t,\bar Y_t)].$
We claim that $\tau_H^i >\tau_N^i$. In fact, by the definition of $\tau_H^i$, we know that
$\tau_H^i\geq\tau_N^i$. If $\tau_H^i=\tau_N^i$, it follows  the definition of $\tau_N^i$ that
$$
Y^{i,N}(\tau_N^i)-\bar Y^{i,N}(\tau_N^i)=0, \ Y_{\tau_N^i}^{i,N}(r)-\bar Y_{\tau_N^i}^{i,N}(r)<0, \ r\in [-r_0,0),
$$
thus, $(Z^{i,N}_{\tau_N^i})^+\equiv0$. Moreover, by Proposition \ref{prop1}, we have $\tau_N\leq \tau^i$, $i=1,\cdot\cdot\cdot, d$, thus
$Y_{\tau_N^i}\leq\bar Y_{\tau_N^i} $  and  $Y^{N}(\tau_N^i)\leq\bar Y^{N}(\tau_N^i). $
Letting $t=\tau_N^i, \xi=Y_{\tau_N^i}^Z(=Y_{\tau_N^i}), \eta=\bar Y_{\tau_N^i}$ in the condition $\textbf{(A1')}$, we have $H^i(\tau_N^i)=H^i(\tau_H^i)< 0.$
This would  be in contradiction with  \eqref{tauH}!
Therefore, the claim holds true. Furthermore,
$H^i(t)\leq 0 $
 for all $t\in [\tau_N^i, \tau_H^i]$. 

Now,
we   take  the following $C^2$-approximation of $s^+$  as   \cite[ Theorem 1.1]{HW}.
For any $n\ge 1$,   construct $\psi_n: \R\to [0,\infty)$  as follows: $\psi_n(s)=\psi_n'(s)=0$ for $s\in (-\infty,0]$, and
$$\psi_n''(s)=\beg{cases} 4n^2s, & s\in [0,\ff 1 {2n}],\\
-4n^2(s-\ff 1 n), & s\in [\ff 1 {2n}, \ff 1 n],\\
0, &\text{otherwise}.\end{cases}$$
One can see that
\beq\label{1.3}
0\le \psi_n'\le 1_{(0,\infty)}, \ \text{and\ as\ } n\uparrow\infty: \ 0\le \psi_n(s)\uparrow s^+,\ \ s\psi_n''(s)\le  1_{(0,\ff 1 n)}(s)\downarrow 0.
\end{equation}
Noting  that
 $\psi_n(Y^{i,N}(0)-\bar Y^{i,N}(0))=\psi_n(\xi^i(0)-\eta^i(0)-N(\xi^i)+N(\eta^i))=0,$
 $\si=\bar\si$,   and
 \begin{align*}
\d (Y^{i,N}(t)-\bar Y^{i,N}(t))=&    (b^i(t,Y_t)-b^i(t,\bar{Y}_t))\d t+ \<h^i(t,Y_t)-\bar{h}^i(t,\bar{Y}_t),\d \<B\>(t)\> \\
 &+ \sum_{j=1}^m (\si^{ij}(t, Y_t)-\si^{ij}(t,\bar{Y}_t))\d B^j(t ),
\end{align*}
we then have from It\^o's  formula that
\beq\label{1.4}
\beg{split}
&\psi_n(Y^{i,N}(t\wedge\tau_H^i)-\bar Y^{i,N}(t\wedge\tau_H^i))^2\\
=&2\sum_{j=1}^m \int_{0}^{t\wedge\tau_H^i} (\si^{ij}(s, Y_s)- \si^{ij}(s,\bar Y_s))\{\psi_n\psi_n'\}(Y^{i,N}(s)-\bar Y^{i,N}(s))\d B^j(s)\\
&+2\int_{0}^{t\wedge\tau_H^i} \{\psi_n\psi_n'\}(Y^{i,N}(s)-\bar Y^{i,N}(s))\<h^i(s,Y_s)-\bar h^i(s,\bar Y_s),\d \<B\>(s)\> \\
&+2\int_{0}^{t\wedge\tau_H^i} (b^i(s,Y_s)-\bar b^i(s,\bar Y_s))\{\psi_n\psi_n'\}(Y^{i,N}(s)-\bar Y^{i,N}(s))\d s\\
&+  \sum_{j=1,k=1}^m \int_{0}^{t\wedge\tau_H^i}\{ \psi_n\psi_n''+\psi_n'^2\}(Y^{i,N}(s)-\bar Y^{i,N}(s))\\
&\qquad\qquad\qquad\times  (\si^{ij}(s,Y_s)- \si^{ij}(s,\bar Y_s))(\si^{ik}(s,Y_s)- \si^{ik}(s,\bar Y_s))\d \<B\>_{jk}(s)\\
 &=M_i(t\wedge\tau_H^i)+\bar{M}_i(t\wedge\tau_H^i)+I_1+I_2
 \end{split}\end{equation}for any $n\ge 1$, $1\le i\le d$ and $t\ge 0,$
where
\begin{align*}
&M_i(t\wedge\tau_H^i):= 2\sum_{j=1}^m \int_{0}^{t\wedge\tau_H^i} (\si^{ij}(s, Y_s)- \si^{ij}(s,\bar Y_s))\{\psi_n\psi_n'\}(Y^{i,N}(s)-\bar Y^{i,N}(s))\d B^j(s),\\
&\bar{M}_i(t\wedge\tau_H^i):=2\int_{0}^{t\wedge\tau_H^i} \{\psi_n\psi_n'\}(Y^{i,N}(s)-\bar Y^{i,N}(s))\<(h^i(s,Y_s)-\bar h^i(s,\bar Y_s)),\d \<B\>(s)\> \\
&\qquad\qquad-4\int_{0}^{t\wedge\tau_H^i} G[\{\psi_n\psi_n'\}(Y^{i,N}(s)-\bar Y^{i,N}(s))(h^i(s,Y_s)-\bar h^i(s,\bar Y_s))]\d s,\\
&I_1:=2\int_{0}^{t\wedge\tau_H^i} (b^i(s,Y_s)-\bar b^i(s,\bar Y_s))\{\psi_n\psi_n'\}(Y^{i,N}(s)-\bar Y^{i,N}(s))\d s\\
&\qquad\qquad+4\int_{0}^{t\wedge\tau_H^i} G[\{\psi_n\psi_n'\}(Y^{i,N}(s)-\bar Y^{i,N}(s))(h^i(s,Y_s)-\bar h^i(s,\bar Y_s))]\d s,\\
&I_2:=\sum_{j=1,k=1}^m \int_{0}^{t\wedge\tau_H^i}\{ \psi_n\psi_n''+\psi_n'^2\}(Y^{i,N}(s)-\bar Y^{i,N}(s))\\
&\qquad\qquad\qquad\times  (\si^{ij}(s,Y_s)- \si^{ij}(s,\bar Y_s))(\si^{ik}(s,Y_s)- \si^{ik}(s,\bar Y_s))\d \<B\>_{jk}(s).
\end{align*}
For simplicity, let $\Phi^n_N(s)=\{\psi_n\psi_n'\}(Y^{i,N}(s)-\bar Y^{i,N}(s))$, $s \in[0,T]$.  Due to $\Phi^n_N(s)\geq 0$, $s \in[0,T]$ and the property (a) of $G$, for any $ n\geq1, \ s \in[0,T]$, we have
\begin{equation}\label{Gp}
G[\{\psi_n\psi_n'\}(Y^{i,N}(s)-\bar Y^{i,N}(s))(h^i(s,Y_s)-\bar h^i(s,\bar Y_s))] =\Phi^n_N(s)G(h^i(s,Y_s)-\bar h^i(s,\bar Y_s)).
\end{equation}
Moreover, note that  $0\le\psi_n'(Y^{i,N}(s)-\bar Y^{i,N}(s))\le 1_{\{Y^{i,N}(s)>\bar Y^{i,N}(s)\}}$. By \eqref{tauH},
it holds that
\begin{equation}\begin{split}\label{le}
H^i(s)\Phi^n_N(s)\leq0,    \ s\in [0, \tau_H^i].
\end{split}\end{equation}
{\bf{Step 1. Estimate the term $I_1$:}}
 Combining \eqref{Gp} and \eqref{le}, {\bf (H1)},   $0\le\psi_n'\le 1$ and properties (b) and (c) of $G$,
 there exists a constant $C(T,\bar{\sigma})> 0  $ depends on $T$ and $\bar{\sigma}$ such that for any $n \geq1,$   $\ t\in[0,T]$
\beg{align*}
I_1&=2\int_{0}^{t\wedge\tau_H^i}[b^i(s,Y_s)-\bar b^i(s,\bar Y_s)+2G(h^i(s,Y_s)-\bar h^i(s,\bar Y_s))]\Phi^n_N(s)\d s\\
& = 2\int_{0}^{t\wedge\tau_H^i}[b^i(s,Y_s)-b^i(s, Y_s^Z)+2G(h^i(s,Y_s)-h^i(s, Y_s^Z)]\Phi^n_N(s)\d s\\
&\quad+2\int_{0}^{t\wedge\tau_H^i} [b^i(s,Y_s^Z)-\bar b^i(s,\bar Y_s)+2G(h^i(s,Y_s^Z)-\bar h^i(s,\bar Y_s))]\Phi^n_N(s)\d s\\
&\leq2\int_{0}^{t\wedge\tau_H^i}[b^i(s,Y_s)-b^i(s,Y_s^Z)+2G(h^i(s,Y_s)-h^i(s,Y_s^Z))]\Phi^n_N(s)\d s\\
&\leq\int_{0}^{t\wedge\tau_H^i}[b^i(s,Y_s)-b^i(s,Y_s^Z)+2G(h^i(s,Y_s)-h^i(s,Y_s^Z))]^2\d s\\
&\qquad +\int_{0}^{t\wedge\tau_H^i}\psi_n(Y^{i,N}(s)-\bar Y^{i,N}(s))^2\d s\\
&\leq\int_{0}^{t\wedge\tau_H^i}C(T,\bar{\sigma})| (Z^{i,N}_s)^+|^2\d s+\int_{0}^{t\wedge\tau_H^i}\psi_n(Y^{i,N}(s)-\bar Y^{i,N}(s))^2\d s\\
&\leq\int_{0}^{t\wedge\tau_H^i}C(T,\bar{\sigma})| (Y_s^{i,N}-\bar Y_s^{i,N})^+|^2\d s.
\end{align*}
{\bf{Step 2.  Estimate the term $I_2$:}}
  From 
  \eqref{1.3}, {\bf (H1)}, Remark \ref{rem1} and \eqref{B}, we get
\beq\label{1.6} \beg{split}
I_2&\leq  \int_{0}^{{t\wedge\tau_H^i}} \left(1_{\{Y^{i,N}(s)-\bar Y^{i,N}(s)\in (0,\ff 1 n)\}}+1_{\{Y^{i,N}(s)-\bar Y^{i,N}(s)\in (0,\infty)\}}\right)\\
 &\qquad\qquad \times \sum_{j=1,k=1}^m(\si^{ij}(s,Y_s)- \si^{ij}(s,\bar Y_s))(\si^{ik}(s,Y_s)- \si^{ik}(s,\bar Y_s))\d \<B\>_{jk}(s) \\
&\leq  \int_{0}^{{t\wedge\tau_H^i}} C(T,\bar{\sigma})|(Y^{i,N}(s)- \bar Y^{i,N}(s))^+|^2\d s,\ \ n\ge 1, \ t\in [0,T].
\end{split}\end{equation}
{\bf{Step 3.  Estimate the term ${M}_i(t\wedge\tau_H^i)$:}}
 Since $M_i$ is  a  $G$-martingale, we have
 $$\bar{\E}^G  {M}_i(t\wedge\tau_H^i)= 0, \ t\in [0, T].$$
{\bf{Step 4.  Estimate the term $\bar{M}_i(t\wedge\tau_H^i)$:}}
Note that  $\bar{M}_i$ is a non-increasing $G$-martingale, we obtain from \eqref{Gp} that  $\bar{\E}^G\bar{M}_i(t\wedge\tau_H^i)\leq  0, \ t\in [0, T].$

Therefore, for any $n\ge 1$ and $t\in [0,T  ]$,  there exists a constant $C>0$ such that
\beg{equation*}\beg{split} 
\bar{\E}^G[ {1}_{A} \psi_n(Y^{i,N}(t\wedge\tau_H^i)-\bar Y^{i,N}(t\wedge\tau_H^i))^2]
&\leq C\bar{\E}^{G} {1}_{A}\int_{0}^{{t\wedge\tau_H^i}}|(Y^{i,N}(s)-\bar Y^{i,N}(s))^+|^2\d s\\
&\le C\int_{0}^{t \wedge\tau_H^i}\bar{\E}^G( {1}_{A}|(Y^{i,N}(s)-\bar Y^{i,N}(s))^+|^2)\d s.
\end{split}
\end{equation*}
 Letting  $n\uparrow\infty$, by the monotone convergence theorem in \cite[Theorem 6.1.14]{peng4}, we have
 \begin{align*}
 & \bar{\E}^G \{ {1}_{A}((Y^{i,N}(t \wedge\tau_H^i)-\bar Y^{i,N}(t \wedge\tau_H^i))^+)^2\}\le C  \int_{0}^{t\wedge\tau_H^i}  \bar{\E}^G[ {1}_{A}|(Y_{s}^{i,N}-  \bar Y_{s}^{i,N})^+|^2]\d s.
 \end{align*}
 By   Gronwall's inequality and \eqref{BDD},  we arrive at
$$
\bar{\E }^G[ {1}_{A}|(Y^{i,N}({t\wedge\tau_H^i})-  \bar Y^{i,N}({t\wedge\tau_H^i}))^+|^2]=0,
\ t\in [0, T].
$$
 This yields that
  $Y^{i,N}({t\wedge\tau_H^i})\leq \bar Y^{i,N}({t\wedge\tau_H^i})$
  on $A=\{\tau_N^i=\tau_N<M\}$ for all    $t\in [\tau_N^i, T]\subset [0,T]$,
 which is contradicts with the definition of $\tau_N^i$ in  \eqref{tau_D} for $\tau_H^i >\tau_N^i$.
 Therefore, we prove that for every $M>0$,
$\mathcal{C}\{\tau_N<M\}=0$, i.e., $\mathcal{C}\{\tau_N=\infty\}=1$.

  In general, if $\textbf{(A1)}$ in Theorem \ref{T1.1} holds, let $\bar{b}_\epsilon=\bar{b}+\vec{\epsilon},$  where $\vec{\epsilon}=(\epsilon,\cdot\cdot\cdot,\epsilon)\in\R^d,$  $\epsilon>0$.
Let $\bar{Y}^\epsilon(t)$ solve \eqref{E1} with  $\bar{Y}^\epsilon_0=\bar{Y}_0$ and the drift term $ \bar{b}_\epsilon $ instead of $b$. It is  easy to deduce under the same conditions in Theorem \ref{T1.1} that
\begin{equation}\label{epsilon}
  \lim_{\epsilon \to 0^+}\bar{\E }^G\sup_{t\in[0,T]}|\bar{Y}^\epsilon(t)-\bar{Y} (t)|=0.
\end{equation}
Note that $\bar{b}_\epsilon $ satisfies condition    $\textbf{(A1')}$, by the above discussion, we have $\mathcal{C}$-q.s.
 $$Y_t\leq_N\bar Y_t^\epsilon,\   t\in[0,T]. $$
Letting $\epsilon\rightarrow 0,$ it follows from \eqref{epsilon} and the continuity of $N$ that
$\mathcal{C}$-q.s.
$\mathcal{C}$-q.s.
 $$Y_t\leq_N\bar Y_t,\   t\in[0,T]. $$
\end{proof}
The following result shows that   the sufficient    conditions  appeared in Theorem \ref{T1.1}  are also necessary if all coefficients are continuous on $[0,\infty)\times\C$.
\beg{thm}\label{T1.2} Let {\bf (H1)}-{\bf (H4)} hold and     $\eqref{E1}$ be $N$-order-preserving. Moreover, if $b, h, \si$ and $\bar b, \bar h,\bar \si$ are continuous on $[0,\infty)\times \C$,
then conditions $\textbf{(A1)}$ and $\textbf{(A2)}$ in Theorem \ref{T1.1} hold.
\end{thm}

\begin{proof}[$\textbf{Proof of (A1)}$] We use the representation theorem  as explained in      \cite[Theorem 3.3]{HYang}.  Let $1\leq i\leq d$ be fixed.
For any $t_0\ge 0$ and $\xi, \eta\in \C$ with $\xi\le_N\eta$ and $\xi^i(0)-N^i(\xi)=\eta^i(0)-N^i(\eta)$, it holds $\mathcal{C}$-q.s.
\begin{align}
\label{com}Y(t_0,\xi)_t\leq_N \bar{Y}(t_0,\eta)_t,\ \ t\geq 0.
\end{align}
 Denote $Y(t)=Y(t_0,\xi;t)$, $\bar Y(t)=\bar Y(t_0,\eta;t)$, $Y^N(t)=Y(t)-N(Y_t)$ and  $\bar{Y}^N(t)=\bar{Y}(t)-N(\bar{Y}_t)$. Recalling the expression \eqref{G(A)},
for any $\gamma\in \mathbb{S}_+^m\bigcap[\underline{\sigma}^2\textbf{I}_{m\times m}, \bar{\sigma}^2\textbf{I}_{m\times m}]$, taking   $\theta_s=\sqrt{\gamma},s\geq 0$, we have $\P_{\gamma}$-a.s. $\<B\>(r)=r\gamma$.  Accordingly, 
by   \eqref{E1}  and \eqref{com}, for any $s\ge 0$, we obtain $\P_\gamma$-a.s.
\beq\label{XA}\beg{split}
0&\ge Y^{i,N}(t_0+s)-\bar Y^{i,N}(t_0+s)   =\xi^i(0)-N(\xi^i)-\eta^i(0)+N(\eta^i)\\
 &\quad +   \int_{ t_0}^{ t_0+s}[b^i(r,Y_r)-\bar b^i(r,\bar Y_r)]\,\d r+\int_{ t_0}^{ t_0+s}\<h^i(r,Y_r)-\bar h^i(r,\bar Y_r),\d\<B\>(r)\> \\
&\quad +  \sum_{j=1}^m \int_{ t_0}^{ t_0+s}
[\sigma^{ij}(r,Y_r)-\bar \sigma^{ij}(r,\bar Y_r)]\,\d B^j(r)\\
&=\int_{ t_0}^{ t_0+s}[b^i(r,Y_r)-\bar b^i(r,\bar Y_r)]\,\d r+ \int_{ t_0}^{ t_0+s}\<h^i(r,Y_r)-\bar h^i(r,\bar Y_r),\gamma\>\d r \\
&\quad +  \sum_{j=1}^m  \int_{ t_0}^{ t_0+s}[\sigma^{ij}(r,Y_r)-\bar \sigma^{ij}(r,\bar Y_r)]\,\d B^j(r).
\end{split}\end{equation}
Letting $\E_{\P_\theta}=\E_{\P_\gamma}$ in \eqref{rep2}, taking expectation in \eqref{XA} under $\P_\gamma$, by Remark \ref{remboun},     we
have
\begin{equation}\begin{split}\label{sf}
 \frac{1}{s}\int_{ t_0}^{ t_0+s}\E_{\P_\gamma}\{[b^i(r,Y_r)-\bar b^i(r,\bar Y_r)]+\< h^i(r, Y_r)-\bar{h}^i(r,\bar{Y}_r),\gamma\>\}\d r\leq 0,\ \ s>0.
\end{split}\end{equation}
Letting $s\downarrow 0$ in \eqref{sf}, it follows from \eqref{rep2}, \eqref{BDD}, the continuity of $b,\bar{b},h,\bar{h}$ and dominated convergence theorem that
\begin{equation*}\begin{split}
[b^i( t_0,\xi)-\bar b^i( t_0,\eta)]+\<h^i( t_0,\xi)-\bar h^i(t_ 0,\eta),\gamma\>\leq 0.
\end{split}\end{equation*}
In terms of  the definition of $G$ in \eqref{G(A)}, we get
\begin{equation*}\begin{split}
&[b^i( t_0,\xi)-\bar b^i( t_0,\eta)]+2G(h^i( t_0,\xi)-\bar h^i( t_0,\eta))\\
=&[b^i( t_0,\xi)-\bar b^i( t_0,\eta)]+2\sup_{\gamma\in \mathbb{S}_+^m\bigcap[\underline{\sigma}^2\textbf{I}_{m\times m}, \bar{\sigma}^2\textbf{I}_{m\times m}]}\frac{1}{2}\<h^i( t_0,\xi)-\bar h^i( t_0,\eta),\gamma\>\leq 0.
\end{split}\end{equation*}
 {Then \textbf{(A1)} holds.}
\end{proof}

\beg{proof}[\textbf{Proof of (A2)}] For any  $\xi,\bar{\xi}\in\C$ with $\xi\leq_N\bar{\xi}$, it holds $\mathcal{C}$-q.s.
$$Y_t\leq_N \bar{Y}_ t, \ \ t\geq  0.$$
Now, letting $\theta_s=\bar{\sigma}$, by \eqref{rep2},  we have $\P_\theta$-a.s.
$Y_t\leq_N \bar{Y}_t,    t\geq  0.$ Noting that $\P_\theta$-a.s. $\<B\>(r)=\bar{\sigma}^2r$ and $\P_{\theta}$ is the law of $\int_0^\cdot \theta_s \d W^0_s$, thus \eqref{E1} comes back to the SDE driven by classic Brownian motion under $\P_\theta$. According to the necessary condition of order preservation for functional SDEs  with distribution independent in  \cite[Theorem 4.8]{HY}, we prove \textbf{(A2)}.
\end{proof}

\end{document}